\documentclass[11pt]{article}

\usepackage{wrapfig}
\usepackage{times}
\usepackage{epsfig}
\usepackage{graphicx}
\usepackage[dvipsnames]{xcolor}
\usepackage{amsmath}
\usepackage{amssymb}
\usepackage{amsthm}
\usepackage{appendix}
\usepackage{bm}
\usepackage{booktabs}
\usepackage{wrapfig}
\usepackage[all]{xy}         
\usepackage[pdfdisplaydoctitle=true,
            colorlinks=true,
            urlcolor=blue,
            citecolor=blue,
            linkcolor=blue,
            pdfstartview=FitH,
            pdfpagemode= UseNone,
            bookmarksnumbered=true]{hyperref}\usepackage{todonotes}
\usepackage[normalem]{ulem}
\usepackage[margin=1in]{geometry}
\usepackage{setspace}
\usepackage[caption = false]{subfig}
\usepackage{tikz}
\usepackage{tkz-euclide}

\usepackage{titlesec}
\titlespacing*{\section}
{0pt}{2ex plus 1ex minus .2ex}{2ex plus .2ex}
\titlespacing*{\subsection}
{0pt}{2ex plus 1ex minus .2ex}{2ex plus .2ex}
\titlespacing*{\subsubsection}
{0pt}{2ex plus 1ex minus .2ex}{1ex plus 0ex}
\titlespacing*{\paragraph}
{0pt}{0ex}{1ex plus 0ex}

\usepackage{multirow}
\usepackage{adjustbox}

\usepackage{multirow}
\usepackage{enumitem}
\setlist[itemize,1]{leftmargin=.15in,parsep=0in}
\setlist[enumerate,1]{leftmargin=.18in,parsep=0in}
\usepackage{comment}
\usepackage{titlesec}
\usepackage{csquotes}
\titleformat{\title}[block]{\Large\bfseries\filcenter}{}{1em}{}
\titleformat{\section}[block]{\raggedright\large\bfseries}{\thesection}{1em}{}
\titleformat{\subsection}[block]{\raggedright\normalfont\bfseries}{\thesubsection}{1em}{}
\titleformat{\subsubsection}[runin]{\raggedright\normalfont\itshape}{\thesubsubsection}{1em}{}

\titleformat*{\section}{\LARGE\bfseries}
\titleformat*{\subsection}{\Large\bfseries}
\titleformat*{\subsubsection}{\large\bfseries}
\titleformat*{\paragraph}{\large\bfseries}
\titleformat*{\subparagraph}{\large\bfseries}

\newcommand{\id}{\mathrm{id}}

\newcommand{\Constanta}{C}

\theoremstyle{plain}
\newtheorem{theorem}{Theorem}
\newtheorem{lemma}[theorem]{Lemma}
\newtheorem{proposition}[theorem]{Proposition}
\newtheorem{conjecture}[theorem]{Conjecture}
\newtheorem{corollary}[theorem]{Corollary}
\newtheorem{remark}[theorem]{Remark}

\newtheorem*{theorem*}{Theorem}

\date{}

\title{Breakdown of smooth solutions to the subcritical EPDiff equation}
\author{Martin Bauer, Stephen C. Preston, and Justin Valletta}

\begin{document}
\tabcolsep 1pt

\maketitle

\begin{abstract}
We consider the EPDiff equation on $\mathbb{R}^n$ with the integer-order homogeneous Sobolev inertia operator $A=(-\Delta)^k$. We prove that for arbitrary radial initial data and a sign condition on the initial momentum, the corresponding radial velocity solution has $C^1$ norm that blows up in finite time whenever $0\le k<n/2+1.$  Our approach is to use Lagrangian coordinates to formulate EPDiff as an ODE on a Banach space, enabling us to use a comparison estimate with the Liouville equation. Along the way we derive the Green function in terms of hypergeometric functions and discuss their properties. This is a step toward proving the general conjecture that the EPDiff equation is globally well-posed for any Sobolev inertia operator of any real order $k$ if and only if $k\ge n/2+1$. 
\end{abstract}

\tableofcontents

\section{Introduction}
\paragraph{The EPDiff equation:} We are interested in a family of non-linear PDEs, referred to as the EPDiff equation, which are given by
\begin{equation}\label{eq:EPDiff}
\Omega_t+\nabla_U \Omega+  (\nabla U)^T\Omega+\operatorname{div}(U)\Omega=0,\qquad \Omega=AU,
\end{equation}
where $U:[0,T)\times \mathbb R^n\to \mathbb R^n$ is a time-dependent vector field, $A$ is an $L^2$ symmetric, continuous linear operator, called the \textit{inertia operator}, and $\Omega$ is often regarded as the momentum. This family of equations first arose in the context of the Camassa-Holm equation~\cite{camassa1993integrable}, which corresponds to the order-one Sobolev inertia operator $A=1-\Delta$; see \cite{holm2005momentum,holm1998semidirect,misiolek1998shallow}. More generally, it encompasses many classic one-dimensional fluid models, such as Burgers', the modified Constantin-Lax-Majda~\cite{constantin1985simple,Okamoto_2008}, and the Hunter-Saxton equation~\cite{hunter1991dynamics}. It can thereby be viewed as an $n$-dimensional generalization of these fluid equations. Moreover, EPDiff admits a geometric interpretation as an Euler-Arnold equation~\cite{arnold1998topological}: it can be realized as the geodesic equation of a right-invariant Riemannian metric on the group of diffeomorphisms; such equations are also called Euler-Poincar\'{e} equations, which is how EPDiff acquired its name. The interest in the EPDiff equations is thus further driven by the fundamental role that right-invariant Sobolev metrics on diffeomorphism groups play in template matching and shape analysis~\cite{dryden2016statistical,younes2019shapes,bauer2014overview}, particularly in the LDDMM framework~\cite{beg2005computing}: in the spirit of Grenander's pattern theory~\cite{grenander1996elements} this approach represents the differences between shapes as optimal diffeomorphisms between objects, where optimality is measured precisely with respect to a right-invariant metric on the diffeomorphism group. The EPDiff equation thereby arises as the first-order optimality condition. In this context it is also referred to as the template matching equation~\cite{hirani2001averaged}. \\
\paragraph{Known results on well-posedness and blowup of solutions:}
Since its introduction by Holm and Marsden at the turn of the century \cite{holm2005momentum}, much effort has been dedicated to studying local and global well-posedness of the EPDiff equation. These investigations largely rest on the seminal ideas of Ebin and Marsden \cite{ebin1970groups} who famously used Arnold's geometric framework \cite{arnold1966} to prove local well-posedness of the incompressible Euler equations. Using similar methods, the local well-posedness of the EPDiff equation has been established in a quite general setting: assuming that the inertia operator is a pseudo-differential operator of order $2k$, local well-posedness has been shown for $k\ge\frac12$, regardless of the dimension $n$; see~\cite{gay2009well,misiolek2010fredholm,escher2014right,bauer2015local,bauer2020well,trouve2005local}. Under stronger assumptions on the order $k$, namely when $k>n/2+1$, then global well-posedness of EPDiff has been shown; see e.g., the work of Escher, Kolev, Michor, Mumford, and others~\cite{bruveris2017completeness,escher2014geodesic,mumford2013euler,misiolek2010fredholm,bauer2015local,bauer2025regularity,ebin1970groups}. For $n=1$ the global existence has been extended to $k=\frac{3}{2}=\frac{n}2+1$, see~\cite{preston2018euler,bauer2020geodesic}.

This raises the question on the existence of breakdown of solutions for EPDiff with an inertia operator of lower order (below the critical index $n/2+1$).   Due to the aforementioned connections to one-dimensional fluid models this has been studied in detail in dimension one for the inertia operator $A=(\sigma-\Delta)^k$ with $\sigma\in\{0,1\}$ and $k\in\{0,1\}$: when $k=0$, then the one-dimensional EPDiff equation reduces to the inviscid Burgers' equation, for which breakdown is well-known; for $\sigma=0$ and $k=1/2$ this corresponds to the modified Constantin-Lax-Majda for which blowup is known~\cite{castro2010infinite,bauer2016geometric,preston2018euler}; for $k=1$ and $\sigma=0$, this corresponds to the Hunter-Saxton equation, for which one can obtain an explicit solution formula \cite{lenells2007hunter,bauer2014homogenous} that leads to a direct proof of breakdown; see also~\cite{yin2004structure}; finally, for $k=1$ and $\sigma=1$, this gives the Camassa-Holm equation, for which breakdown is long known~\cite{camassa1993integrable,constantin1998wave}; see also the work of McKean~\cite{mckean2015breakdown}, who obtained the complete picture of the breakdown mechanism for this equation. Combining these blowup results with the previously discussed global existence results, one obtains a complete characterization of global well-posedness (existence of blowup, resp.) for EPDiff with integer-order inertia operator in dimension one.

In higher dimensions, the investigation of solution breakdown for the EPDiff equation was only explored much more recently by Chae and Liu~\cite{chae2012blow}, where they confirmed breakdown for the higher-dimensional Burgers' equation. Shortly after, breakdown for the higher-dimensional Camassa-Holm equation, corresponding to $k=1$, was established by Li, Yu, and Zhai \cite{li2013euler}. In their work they showed that the one-dimensional breakdown mechanism can be adapted to radial solutions in higher dimensions. Here, we emphasize that both of these results have in common, that already the one-dimensional equation admits solutions that break down. In our recent work~\cite{bauer2024liouville} we have shown that the EPDiff equation corresponding to $A=(\sigma-\Delta)^2$ in dimension $n\geq 3$ breaks down in finite time, in both the homogeneous $(\sigma=0)$ and the nonhomogeneous $(\sigma=1)$ cases. This was the first instance where the breakdown is a purely higher-dimensional phenomenon, i.e., the one-dimensional equation for $k=2$ exists globally in time. 

Summarizing the previous paragraphs, the known cases here are: global well-posedness for any real $k>n/2+1$ and any $n$ when $\sigma>0$ and for $k=3/2$ when $n=1$. Solution breakdown has been obtained for any $\sigma\ge 0$ for $k\in\{0,1\}$ and any $n$; for $k=2$ and $n\ge 3$; and $k=\tfrac{1}{2}$ when $n=1$. As a consequence of these results it has been conjectured that the index $k=\frac{n}2+1$ is indeed critical for this property, i.e.: 
\begin{conjecture}\label{conjecture}
    Let $A$ be an elliptic pseudo differential operator of order $2k$ acting on $C^{\infty}(\mathbb R^n,\mathbb R^n)$, that is symmetric and positive w.r.t. the $L^2$-inner product. Then the corresponding EPDiff equation is globally well-posed for all smooth initial data $u_0$ if and only if $k\ge n/2+1$.
\end{conjecture}
 Note that the particular case $A=(\sigma-\Delta)^k$ with $k< \frac{n}{2}+1$ in the above conjecture remains open in most cases. The objective of the present paper is to add further evidence for this conjecture by proving new solution breakdown results.\\

\paragraph{Main Contributions:}
The main result of the present article is the following theorem, which establishes blowup for the homogenous inertia operator $A=(-\Delta)^k$ for any integer $k$ below the critical threshold:
\begin{theorem*}[Theorem~\ref{theorem:maintheorem} in Section~\ref{breakdown}]
    Let $k$ be an integer with $0\le k<n/2+1$. Then there is radial initial data $u_0\in H^\infty(\mathbb{R}^n,\mathbb{R}^n)$ such that the corresponding radial solution to the $n$-dimensional EPDiff equation \eqref{eq:EPDiff} with the homogeneous Sobolev inertia operator $A=(-\Delta)^k$ has $C^1$ norm that blows up in finite time. 
\end{theorem*}
This result will be established using a refinement of the approach developed in our previous paper~\cite{bauer2024liouville}, see Theorem~\ref{generalbreakdown}.
In this theorem, 
we derive conditions on the Green function $\delta$ of the solution to $A(u(r)\partial_r)=\omega(r)\partial_r$ that imply $C^1$ blow-up of the radial velocity field $u$. Note that this theorem is a modified version of Theorem B in  \cite{bauer2024liouville} with hypotheses that are less restrictive and at the same time easier to verify. In the same part, Section~\ref{sec:dimensionblowup}, we also present a second breakdown criteria for solutions to EPDiff based on reducing the dimension to the smallest dimension in which breakdown occurs. This provides, for example, simple proofs of breakdown to the higher-dimensional Burgers', Hunter-Saxton, and modified Constantin-Lax-Majda equations; see Corollary~\ref{burgersCH}. Lastly, Section~\ref{breakdown} is devoted to the proof of our main result stated previously, namely the breakdown of solutions to EPDiff with the homogeneous Sobolev inertia operator. For this we will need several results on hypergeometric functions and an expression for the Green function of the operator $A=(-\Delta)^k$.\\ 
\paragraph{Future Directions:} Building on the findings of this paper, two immediate directions for future research emerge: extensions of the breakdown results to non-integer $k$ on the one hand and to the non-homogenous inertia operator $A=(1-\Delta)^k$ on the other hand. These endeavors come with a new set of difficulties on which we will briefly comment in Section~\ref{sec:future}. \\

\paragraph{Acknowledgements and Data availability statement:}  MB and JV were partially funded by BSF grant 2022076.
MB was partially funded by NSF grant CISE 2426549. Data sharing not applicable to this article as no datasets were generated or analysed during the current study.

\section{Background}\label{background}
In this section, we describe the essential background material for the results of this article. In particular, we  recall the setting in which the EPDiff equation is an Euler-Arnold equation, along with the conservation law associated with this equation.

\subsection{The EPDiff equation as an Euler-Arnold equation}\label{EPDiffEA}

Following the presentation in \cite{bauer2015local}, we recall the setting in which the EPDiff equation is an Euler-Arnold equation on an appropriate diffeomorphism group. We consider the space of diffeomorphisms that differ from the identity by a smooth Sobolev function, namely
\[\operatorname{Diff}(\mathbb{R}^n):=\{\id+f\,|\,f\in H^\infty(\mathbb{R}^n,\mathbb{R}^n)\,\text{and}\,\det(\id+df)>0\},\]
where
\[H^\infty(\mathbb{R}^n,\mathbb{R}^n)=\bigcap _{q\ge0}H^q(\mathbb{R}^n,\mathbb{R}^n).\]
The space $\operatorname{Diff}(\mathbb{R}^n)$ is  a regular Fr\'{e}chet Lie group with its Lie algebra being the space of $H^\infty$ vector fields, which we identify with $H^\infty(\mathbb{R}^n,\mathbb{R}^n)$, i.e. $T_e\operatorname{Diff}(\mathbb{R}^n)=\mathfrak{X}_{H^\infty}(\mathbb{R}^n)\cong H^\infty(\mathbb{R}^n,\mathbb{R}^n)$ \cite{hermas2010existence,michor2013zoo}. To define a right-invariant Riemannian metric on the diffeomorphism group, it suffices to prescribe an inner product on the Lie algebra $H^\infty(\mathbb{R}^n,\mathbb{R}^n)$. For this we introduce an $L^2$-symmetric, positive-definite linear operator $A:H^\infty(\mathbb{R}^n,\mathbb{R}^n)\to H^\infty(\mathbb{R}^n,\mathbb{R}^n)$, called the \textit{inertia operator}, which induces an inner product given at the identity by
\[\langle U_1,U_2 \rangle_{A}:=\int_{\mathbb{R}^n}(AU_1\cdot U_2)\,dx,\qquad U_1,U_2\in H^\infty(\mathbb{R}^n,\mathbb{R}^n).\]
This is then extended to a right-invariant Riemannian metric on the diffeomorphism group via right translations:
\[g^A_\eta(u,v):=\langle u\circ\eta^{-1},v\circ\eta^{-1} \rangle_{A},\qquad u,v\in T_{\eta}\operatorname{Diff}(\mathbb{R}^n).\]
One may now define the kinetic energy of a path of diffeomorphisms $\eta:[0,1]\to\operatorname{Diff}(\mathbb{R}^n)$ via 
\[E(\eta)=\frac{1}{2}\int_0^1g^A_\eta(\dot{\eta},\dot{\eta})\,dt.\]

Geodesics with respect to the Riemannian metric $g_\eta^A$ are critical points of the kinetic energy $E(\eta).$ For a right-invariant metric on a Lie group, it is convenient to introduce the Eulerian velocity $U(t,x):=\partial_t\eta(t,\eta^{-1}(t,x))\in H^\infty(\mathbb{R}^n,\mathbb{R}^n).$ Using this change of coordinates, the geodesic equation takes the form 
\begin{equation}\label{EA}
\partial_t\eta(t,x)=U(t,\eta(t,x)),\qquad \partial_tU(t,x)+\operatorname{ad}_{U(t,x)}^\top U(t,x)=0,
\end{equation}
where $\operatorname{ad}_U^T$ is the formal adjoint of the operator $\operatorname{ad}_U$ with respect to the inner product $\langle .\,,.\rangle_{H^k}$. The first order equation in $U$ on the Lie algebra is called the \textit{Euler-Arnold equation}. It was first derived for finite-dimensional Lie groups by Poincar\'{e} \cite{poincare1901forme} (hence the name EPDiff, for Euler-Poincar\'{e} equation on the diffeomorphism group), and was subsequently extended by Arnold \cite{arnold1966} to the infinite-dimensional setting. Following this approach, one can show that the EPDiff equation \eqref{eq:EPDiff} is the Euler-Arnold equation on $\operatorname{Diff}(\mathbb{R}^n)$ with respect to the right-invariant metric induced by the inertia operator $A$; see also \cite{misiolek2010fredholm,bauer2015local}. If $\eta$ is the Lagrangian flow of the vector field $U$, then
\[\frac{d}{dt}\left(\operatorname{Ad}^\top_\eta U\right)=\operatorname{Ad}^\top_\eta\left( U_t+\operatorname{ad}^\top_UU\right),\]
from which one can see that the Euler-Arnold equation \eqref{EA} implies the momentum conservation law
\begin{equation}\label{conservation}
\operatorname{Ad}^\top_{\eta(t)}U(t)=U_0,
\end{equation}
where $U(0)=U_0$. This allows us to eliminate $U(t)$ in \eqref{EA} to get an equation directly on the diffeomorphism group via
\begin{equation}\label{reduction}
\frac{d\eta}{dt}=U(t)\circ\eta(t)=\operatorname{Ad}^\top_{\eta(t)^{-1}}U_0\circ\eta(t),\qquad \eta(0)=\id.
\end{equation}

If $A$ is a sufficiently strong differential operator, then the right side of \eqref{reduction} is smooth as a function of $\eta$ in the Sobolev space $H^s$ for sufficiently large $s$, and we can thereby prove local well-posedness using Picard iteration for any fixed $U_0\in H^s.$ When applied to the Euler equations for a perfect fluid (i.e. on the volume-preserving diffeomorphism group), this technique is called the \textit{particle-trajectoy method} in Majda-Bertozzi \cite{majda_bertozzi_2001}; see also \cite{ebin1984concise,bauer2016geometric}. We shall apply this method to study global existence in the special case when $U_0$ is a smooth, purely radial vector field $U_0=u_0(r)\partial_r$. We first check that radial initial data remains radial during the time evolution of the EPDiff equation. Then we state the conservation law associated with the radial EPDiff equation \eqref{mainomega}.

\begin{lemma}[Radial solutions]\label{lem:radial_solutions}
Let $U_0$ be a purely radial vector field and suppose the inertia operator $A$ preserves radial vector fields. If $U_0=u_0(r)\partial_r$ is the initial velocity for a solution $U(t,r)$ of the EPDiff equation \eqref{eq:EPDiff} defined on its maximal interval of existence $J$, then $U=u(t,r)\partial_r$ is a radial velocity field for each $t\in J.$ Moreover, the radial function $u$ satisfies the radial EPDiff equation 
\begin{equation}\label{mainomega}
\omega_t + u \omega_r + 2u_r \omega + \frac{n-1}{r} u \omega = 0,\qquad AU=\omega\partial_r.
\end{equation}
\end{lemma}

\begin{proof}
Substitute the radial solution $U=u(t,r)\partial_r$ into \eqref{eq:EPDiff} and use the fact that $A$ preserves radial vector fields.
\end{proof}

\begin{remark}\label{vectorlaplacian}
Note that the vector Laplacian acts on radial vector fields by the formula
    \begin{equation}
\Delta\big(u(r)\partial_r\big)=\Big(u''(r)  + \frac{n-1}{r}\, u'(r) - \frac{n-1}{r^2} \, u(r)\Big)\partial_r.
\end{equation}
Consequently  any inertia operator defined in terms of the vector Laplacian $\Delta$ preserves radial vector fields in the sense of Lemma~\ref{lem:radial_solutions}. 
\end{remark}

Next we introduce the radial Lagrangian flow map $\gamma(t,r)$ associated to the component function $u$ of the radial vector field $U=u(t,r)\partial_r$:
\begin{equation}\label{radialflow}
    \frac{\partial\gamma}{\partial t}(t,r)=u(t,\gamma(t,r)),\qquad \gamma(0,r)=r.
\end{equation}
Calculating the group adjoint transpose $\operatorname{Ad}^\top_{\gamma}U$ for radial vector fields $U$ yields via equation \eqref{conservation} a conservation law associated with the radial EPDiff equation. On the diffeomorphism group $\operatorname{Diff}(\mathbb{R}^n)$, this takes the following form.

\begin{proposition}[Conservation law; Lemma 3.6 in \cite{bauer2024liouville}]
\label{conservationlaw}
    Let $U=u(t,r)\partial_r$ with $AU=\omega(t,r)\partial_r$. If $u$ and $\omega$ solve the radial EPDiff equation \eqref{mainomega} on the time interval $[0,T)$ for all $r\ge 0,$ with the flow $\gamma(t,r)$ defined by \eqref{radialflow}, then for all $t\in[0,T)$ and $r\in[0,\infty)$, the following conservation law holds.
    \begin{equation}\label{eqn:conservationlaw}
    \gamma(t,r)^{n-1}\gamma_r(t,r)^2\omega(t,\gamma(t,r))=r^{n-1}\omega_0(r).
    \end{equation}
\end{proposition}

\subsection{Momentum transport formulation}

We will use the conservation law of Proposition \ref{conservationlaw} to express the flow equation $\gamma_t=u\circ\gamma$ in integral form. This amounts to employing \eqref{reduction} from the general Euler-Arnold theory. The results here were established in our previous work \cite{bauer2024liouville}, but are essential to understand the methods used in the present article.

\begin{proposition}[Proposition 3.8 in \cite{bauer2024liouville}]
  Suppose that the inertia operator $A$ is invertible and that the solution of $A(u(r)\partial_r)=\omega(r)\partial_r$ is given by an integral formula of the form
    \begin{equation}\label{intform}
    u(t,r)=\int_0^r\delta(s,r)s^{n-1}\omega(t,s)\,ds+\int_r^\infty\delta(r,s)s^{n-1}\omega(t,s)\,ds,
    \end{equation}
    where the kernel $\delta$ is $C^1$ on $D = \{ (r,s) \, \vert\, \infty\ge s\ge r > 0\}\subset\mathbb{R}^2$. Let $u(t,r)$ be a solution of the radial EPDiff equation \eqref{mainomega} with $u(0,r)=u_0(r)$ and $\omega_0(r)\partial_r=A(u_0(r)\partial_r)$. If $z_0(r):=r^{n-1}\omega_0(r)$, then the flow $\gamma(t,r)$ satisfies
    \begin{equation}\label{gamma}
        \frac{\partial\gamma}{\partial t}(t,r)=\int_0^r\frac{\delta(\gamma(t,s),\gamma(t,r))}{\gamma_s(t,s)}z_0(s)\,ds+\int_r^\infty\frac{\delta(\gamma(t,r),\gamma(t,s))}{\gamma_s(t,s)}z_0(s)\,ds,
    \end{equation}
    and its spatial derivative satisfies
        \begin{equation}\label{rho}
        \frac{\partial}{\partial t}\ln\left(\gamma_r(t,r)\right)=\int_0^r\frac{\partial_2\delta(\gamma(t,s),\gamma(t,r))}{\gamma_s(t,s)}z_0(s)\,ds+\int_r^\infty\frac{\partial_1\delta(\gamma(t,r),\gamma(t,s))}{\gamma_s(t,s)}z_0(s)\,ds
    \end{equation}
\end{proposition}

The system \eqref{gamma}-\eqref{rho} can be written in the form of an autonomous vector field ODE on a Banach space, namely

\begin{equation}\label{gamma1}
        \frac{d\gamma}{dt}(r)=\int_0^r\frac{\delta(\gamma(s),\gamma(r))}{\rho(s)}z_0(s)\,ds+\int_r^\infty\frac{\delta(\gamma(r),\gamma(s))}{\rho(s)}z_0(s)\,ds,
    \end{equation}
        \begin{equation}\label{rho1}
        \frac{d\rho}{dt}(r)=\rho(r)\int_0^r\frac{\partial_2\delta(\gamma(s),\gamma(r))}{\rho(s)}z_0(s)\,ds+\rho(r)\int_r^\infty\frac{\partial_1\delta(\gamma(r),\gamma(s))}{\rho(s)}z_0(s)\,ds,
    \end{equation}
where $\rho(t,r)=\gamma_r(t,r),$ but we treat it as a separate variable to get a closed ODE system. Of course $\gamma$ and $\rho$ are not independent, but since $\gamma$ is the unique antiderivative of $\rho$ such that $\gamma(0)=0,$ we may consider the system as a single ODE for $\rho$ alone. In this sense, we may view \eqref{gamma1} as a consequence of \eqref{rho1}, although it is still convenient at times to treat both equations simultaneously. In fact we will combine these two integro-differential equations into one for our main breakdown results. 

For a fixed function $z_0$,  equation \eqref{rho} for $\rho$ makes sense on the space of bounded positive functions. 
More precisely, denote by $\mathcal{P}$ the space of continuous, bounded positive functions on $[0,\infty)$, i.e.,
\begin{equation}\label{positivecontinuous}
\mathcal{P}=\{\rho\in C([0,\infty),\mathbb{R}^+\,|\,\exists b\ge a>0\text{ such that } \rho(r)\in[a,b]\,\forall r\ge0\}.
\end{equation}
Let $\Gamma$ be the map from $\mathcal{P}$ to $C^1$ diffeomorphisms sending $\rho\mapsto\gamma$. Then we can express equations \eqref{gamma1}--\eqref{rho1} as a single vector field on $\mathcal{P}$.

We can prove local existence of this system with the homogeneous Sobolev inertia operator under the assumption that $z_0$ (or equivalently $\omega_0$) is in a certain weighted $L^1$ space;
see \cite{bauer2024liouville}. We thereby obtain local solutions $\rho$ in the space of continuous positive functions on $[0,\infty),$ and hence local existence of $C^1$ solutions $\gamma.$

We now present the breakdown mechanism employed in this paper. Recall that a classical solution $U:[0,1]\times\mathbb{R}^n\to\mathbb{R}^n$ exists so long as $U$ remains spatially $C^1$, corresponding to $u(t,r)$ being $C^1$ in $r\ge0$ for each $t\in[0,T).$ We will prove breakdown by showing that the radial Lagrangian flow $\gamma$ leaves the diffeomorphism group in finite time, which happens only when $u_r(t,r)$ blows up.

\begin{proposition}[Lemma 3.7 from \cite{bauer2024liouville}]\label{breakdownmech}
        Suppose $u:[0,T)\times[0,\infty)\to\mathbb R$ is a classical solution to the radial EPDiff equation \eqref{mainomega} that vanishes as $r\to\infty.$ It holds that:
    \begin{enumerate}
        \item The Lagrangian flow, defined by $\gamma_t(t,r)=u(t,\gamma(t,r))$ with $\gamma(0,r)=r$ exists on the same time interval as the solution $u$ and is $C^1$ in space and time. 
        \item If $\lim_{t\nearrow T}\gamma_r(t,r)=0$ for some $r\ge0$ and $T>0$, then $u$ cannot be extended as a $C^1$ solution to time $T$.
    \end{enumerate}
\end{proposition}

\section{Two breakdown criteria for the EPDiff equation}\label{sec:generalbreakdown}

In this section, we present two general breakdown results that are in principle applicable to EPDiff equations with various inertia operators, not only those of Sobolev type: first, in Theorem~\ref{generalbreakdown}, we will present a slightly stronger version of the comparison theory based theorem of~\cite{bauer2024liouville}. Second, in Theorem \ref{breakdowndim}, we present a breakdown result for the higher-dimensional EPDiff equation that allows one to reduce breakdown to the smallest dimension where breakdown occurs. This result directly leads to a new and simpler proof for  breakdown of the higher-dimensional Burgers' equation and the higher-dimensional Camassa-Holm equation, as described in Corollary \ref{burgersCH}.

\subsection{Breakdown criteria for radial solutions}\label{sec:radialblowup}

The following theorem is a modified version of the main breakdown criteria (Theorem B) provided in \cite{bauer2024liouville}. The hypotheses here are weaker and seem to be easier to verify in certain situations, such as the situation herein for the higher-order homogeneous Sobolev inertia operator. 

\begin{theorem}\label{generalbreakdown}
Let $u_0$ be initial conditions  such that $\omega_0(r)=Au_0(r)\leq 0$ for all $r>0$, and let the solution to $A(u(r)\partial_r)=\omega(r)\partial_r$ be given by an integral formula of the form \eqref{intform} for a kernel $\delta$ that is smooth and positive
on $D=\{(r,s)\, \vert s\ge r\ge 0\} \backslash\{(0,0)\}$,  
such that for all $(r,s)\in D$, there exists a function $Q:[0,\infty)\to(0,\infty)$ and a constant $C>0$ that satisfies
\begin{equation}\label{hypothesis1}
     \displaystyle{\frac{Q'(r)}{Q(r)}\delta(s,r)+\partial_2\delta(s,r)\geq0} \quad\text{for all}\quad s\in[0,r),
\end{equation}
\begin{equation}\label{hypothesis2}
     \displaystyle{\frac{Q'(r)}{Q(r)}\delta(r,s)+\partial_1\delta(r,s)\geq\frac{C}{Q(s)}} \quad\text{for all}\quad s\in[r,\infty).
\end{equation}
Assume in addition that there exists a local solution $(\gamma,\rho)$ to equations \eqref{gamma1}-\eqref{rho1} with initial condition $u_0$. Then the solution $\rho(t,r)=\gamma_r(t,r)$ reaches zero in finite time $T$ and thus the $C^1$ norm of $u(T,\cdot)$ blows up  in the sense of Proposition~\ref{breakdownmech}.
\end{theorem}

\begin{proof}
    In the following computation we suppress the dependence on $t$. Using the system \eqref{gamma1}-\eqref{rho1}, along with the assumption $\omega_0\leq0,$ we compute that
  \begin{equation}\label{eq:trueequation}
    \begin{aligned}
        \frac{d}{dt}\ln\left(Q(\gamma)\rho\right)&=\frac{Q'(\gamma(r))}{Q(\gamma(r))}\frac{d\gamma}{dt}+\frac{d}{dt}\ln\rho\\
        &=-\int_0^r\left(
        \frac{Q'(\gamma(r))}{Q(\gamma(r))}\delta(\gamma(s),\gamma(r))+\partial_2\delta(\gamma(s),\gamma(r))\right)\frac{s^{n-1}|\omega_0(s)|}{\rho(s)}\,ds\\
        &\qquad-\int_r^\infty\left(
        \frac{Q'(\gamma(r))}{Q(\gamma(r))}\delta(\gamma(r),\gamma(s))+\partial_1\delta(\gamma(r),\gamma(s))\right)\frac{s^{n-1}|\omega_0(s)|}{\rho(s)}\,ds.
    \end{aligned}
    \end{equation}
By \eqref{hypothesis1}, the integral over $[0,r]$ has an everywhere positive integrand. This along with \eqref{hypothesis2} implies that the quantity $q(t,r)=Q(\gamma(t,r))\rho(t,r)/Q(r)$ satisfies the differential inequality 
\begin{equation}\label{blowupinequality}
    \frac{\partial}{\partial t}\ln\left(q(t,r)\right)\leq-C\int_r^\infty\frac{|z_0(s)|/Q(s)}{q(t,s)}\,ds.
\end{equation}
Since this inequality holds on the Banach space of positive continuous functions defined by \eqref{positivecontinuous}, we can now apply standard ODE comparison theorems, see e.g. Lemma 2.3 in \cite{bauer2024liouville} for a version of such a theorem adapted to the present context. Consequently the solution to~\eqref{eq:trueequation} is upper bounded by the solution to the ODE corresponding to equality in \eqref{blowupinequality}. This ODE is precisely the Liouville equation \cite{liouville1853equation}, which admits an explicit solution formula and reaches zero in finite time, see e.g. Lemma 2.1 in \cite{bauer2024liouville} for the explicit solution formula with the precise boundary conditions used in the current situation. An application of the aforementioned ODE comparison theorem thereby implies that the solution $q(t,r)$ to~\eqref{eq:trueequation} approaches zero in finite time. Since our function $Q$ is assumed to be nonzero everywhere, this can only happen if $\rho(T,r)=\gamma_r(T,r)=0$ for some $T>0,$ which by Proposition \ref{breakdownmech} implies that the $C^1$ norm of $u$ approaches infinity as $t\nearrow T.$
\end{proof}

\begin{remark}\label{canonicalQ}
Suppose the kernel $\delta$ is of the form $\delta(r,s)=rs\varphi(r,s),$ where $\varphi$ is smooth and positive
on $D=\{(r,s)\, \vert s\ge r\ge 0\} \backslash\{(0,0)\}$. To apply Theorem~\ref{generalbreakdown} one needs to construct an unknown function $Q$, that 
serves as the basis for the comparison theory. In our experience the choice
\begin{equation}
Q(r):=\frac{1}{r\varphi(0,r)},\qquad Q(0):=\lim_{r\to0^+}\frac{1}{r\varphi(0,r)}.   
\end{equation}
seems to be generally a good candidate for this. In this case the inequalities \eqref{hypothesis1} and \eqref{hypothesis2} read as:
\begin{equation}\label{condition1}
\varphi(0,r)\frac{\partial\varphi}{\partial r}(s,r)-\frac{\partial\varphi}{\partial r}(0,r)\varphi(s,r)\geq0\quad\text{for all}\quad s\in[0,r),
\end{equation}
\begin{equation}\label{condition2}
    \varphi(0,r)\frac{\partial\varphi}{\partial r}(r,s)-\frac{\partial\varphi}{\partial r}(0,r)\varphi(r,s)\geq\frac{C\varphi(0,s)\varphi(0,r)}{r} \quad\text{for all}\quad s\in[r,\infty).
\end{equation}
\end{remark}

\subsection{Breakdown criteria by reducing the dimension}\label{sec:dimensionblowup}
For the EPDiff equation with the Sobolev inertia operator of order $k$, we expect breakdown for every dimension $n$ satisfying $k<n/2+1$. Thus, if breakdown occurs in dimension $n,$ it ought to also occur in any larger dimension. The theorem presented here embodies this observation; it reduces breakdown of the EPDiff equation to the smallest dimension where breakdown first occurs. For this to work, the inertia operator $A$ must satisfy some sort of compatibility condition with the dimension of the vector-valued function on which it acts. Specifically, we need
\begin{equation}\label{compatibledim}
    A_{\mathbb R^{n+1}}\left(u(x_1,\dots,x_n),0\right)=\left(\left(A_{\mathbb R^{n}}u\right)(x_1,\dots,x_n),0\right)\qquad\text{for all $u\in H^{\infty}(\mathbb R^n,\mathbb R^n)$},
\end{equation}
which is easily seen to hold true for any inertia operator that is defined in terms of the Laplacian, such as the integer-order Sobolev inertia operators $A=(\sigma-\Delta)^k$. This also holds for the \textit{fractional} Laplacian $A=(-\Delta)^s$ with $s\in(0,1)$, since it is a Fourier multiplier defined implicitly by
\[\widehat{(-\Delta)^sf}(\xi):=\vert\xi\vert^{2s}\hat{f}(\xi),\]
and the Fourier transform and its inverse both satisfy the compatibility condition \eqref{compatibledim}.

\begin{theorem}\label{breakdowndim}
For any $n\geq 1$, consider the EPDiff equation \eqref{eq:EPDiff} on $\mathbb R^n$ with inertia operator $A=A_{\mathbb R^n}$. Suppose that $A$ is compatible with the dimension in the sense of equation \eqref{compatibledim}. If there exists smooth initial conditions $u_0\in C^{\infty}(\mathbb R^n,\mathbb R^n)$ such that the $n$-dimensional EPDiff equation breaks down in finite time, then there exists smooth initial conditions $\tilde{u}_0\in C^{\infty}(\mathbb R^{n+1},\mathbb R^{n+1})$ such that the $(n+1)$-dimensional EPDiff equation breaks down in finite time. 
\end{theorem}

\begin{proof}
Suppose $u:\mathbb{R}^n\to\mathbb{R}^n$ solves the $n$ dimensional EPDiff equation with initial condition $u_0.$ Define
$\tilde{u}:\mathbb{R}^{n+1}\to\mathbb{R}^{n+1}$ by 
\[\tilde{u}(x_1,\dots,x_{n+1}):=\left(u(x_1,\dots, x_n),0\right).\]
If $A$ is compatible with the dimension in the sense of equation \eqref{compatibledim}, then we can write $m=A_{\mathbb{R}^n}u=(m_1,\dots,m_n)$ and define 
 \[\tilde{m}:=A_{\mathbb{R}^{n+1}}\tilde{u}=\left(A_{\mathbb{R}^n}u,0\right)=(m_1,\dots,m_n,0).\]
   In Euclidean space, $\nabla_{\tilde{u}}\tilde{m}=(\nabla\tilde {m})^T\tilde{u},$ and from the definitions above, we have 
    \[
(\nabla \tilde{m})^T\tilde{u}=
    \begin{pmatrix}
\frac{\partial m_1}{\partial x_1} & \dots & \frac{\partial m_1}{\partial x_n} & \frac{\partial m_1}{\partial x_{n+1}}\\
\vdots & \ddots & \vdots & \vdots\\
\frac{\partial m_n}{\partial x_1} & \dots & \frac{\partial m_n}{\partial x_n} & \frac{\partial m_n}{\partial x_{n+1}}\\
0 & \dots & 0 & 0
\end{pmatrix}
\begin{pmatrix}
    u_1 \\
    \vdots \\
    u_n \\
    0 
\end{pmatrix}
=
\begin{pmatrix}
(\nabla m)^T & 0 \\
0 & 0 \\
\end{pmatrix}
\begin{pmatrix}
    u \\
    0 
\end{pmatrix}.
    \]
Similarly, it holds that
    \[
(\nabla \tilde{u})^T\tilde{m}
=
\begin{pmatrix}
(\nabla u)^T & 0 \\
0 & 0 \\
\end{pmatrix}
\begin{pmatrix}
    m \\
    0 
\end{pmatrix}
    \]
Note also 
\[\operatorname{div}(\tilde{u})\tilde{m}=\operatorname{div}(u)
\begin{pmatrix}
    m \\
    0 
\end{pmatrix},\qquad
\tilde{m}_t=
\begin{pmatrix}
    m_t \\
    0 
\end{pmatrix}
\]
Therefore, if $u$ is a solution to the $n$-dimensional EPDiff equation and breaks down with initial condition $u_0$, then $\tilde{u}$ solves the $(n+1)$-dimensional EPDiff equation and will break down with initial condition $\tilde{u}_0$.
\end{proof}

The previous theorem leads to a new and significantly simpler proof of blow-up for both the higher-dimensional Burgers' equation, originally established in \cite{chae2012blow}, and the higher-dimensional Camassa-Holm equation, originally established in \cite{li2013euler}. Moreover, it leads to a new blow-up result for the higher-dimensional modified Constantin-Lax-Majda equation; the one-dimensional version was shown to have blow-up solutions in \cite{castro2010infinite,bauer2016geometric}. 

\begin{corollary}\label{burgersCH}
There are smooth initial data such that the following members of the EPDiff family break down in finite time.
\begin{itemize}
\item The higher-dimensional Burgers' equation, which corresponds to EPDiff with inertia operator $A=\operatorname{id}$.
\item The higher-dimensional Camassa-Holm equation, which corresponds to EPDiff with inertia operator $A=\operatorname{id}-\Delta$.
\item The higher-dimensional modified Constantin-Lax-Majda equation, which corresponds to EPDiff with inertia operator $A=(-\Delta)^{1/2}.$
\end{itemize}
\end{corollary}

\begin{proof} 
 Breakdown of the one-dimensional Burgers' equation can be shown straightforwardly using the method of characteristics, and the exact time at which breakdown occurs can be derived. Breakdown of the Camassa-Holm equation was known since its original derivation in \cite{camassa1993integrable}, and was rigorously established in \cite{constantin1998global}. Breakdown of the one-dimensional modified Constantin-Lax-Majda equation was established in \cite{castro2010infinite}, and the class of initial data leading to blow-up solutions was then extended in \cite{bauer2016geometric}. Iterative applications of Theorem \ref{breakdowndim} to the breakdown of these one-dimensional equations then yield breakdown in every dimension $n\ge1.$ 
\end{proof}

\begin{remark}\label{infiniteenergy}
The breakdown solution $\tilde{u}$ supplied by Theorem \ref{breakdowndim} belongs only to $C^\infty(\mathbb{R}^n,\mathbb{R}^n),$ as opposed to $H^\infty(\mathbb{R}^n,\mathbb{R}^n)$ or $C_c^\infty(\mathbb{R}^n,\mathbb{R}^n).$ This is because $\tilde{u}$ is constant in $x_{n+1}$, and hence integration over $\mathbb{R}$ with respect to this variable will not yield a finite quantity. 
\end{remark}

\section{Breakdown of smooth solutions for the EPDiff with homogeneous Sobolev inertia operator}\label{breakdown}
We are now ready to  formulate the main result of the present paper:
\begin{theorem}\label{theorem:maintheorem}
    Let $0\le k<n/2+1$. Then there is radial initial data $u_0\in H^\infty(\mathbb{R}^n,\mathbb{R}^n)$ such that the corresponding radial solution $u$ to the $n$-dimensional EPDiff equation \eqref{eq:EPDiff} with the homogeneous Sobolev inertia operator $A=(-\Delta)^k$ has $C^1$ norm that blows up in finite time. 
\end{theorem}
The remainder of this section will be dedicated to the proof of this result: first, in Section~\ref{sec:hypergeometric} we will collect several results on hypergeometric functions, which we will then use in Section~\ref{sec:greensfunction}, to derive a formula for the Green function of (powers of) the Laplacian. Finally in Section~\ref{sec:blowup_hom} we will apply our comparison breakdown criteria for radial solutions from Section \ref{sec:radialblowup} to establish Theorem~\ref{theorem:maintheorem}.

\subsection{Hypergeometric functions}\label{sec:hypergeometric}
As we will see, the Green function for the homogeneous Sobolev inertia operator naturally takes the form of a hypergeometric function. We therefore collect some basic facts about such functions, see~\cite{andrews1999special} for a more comprehensive treatise. First we introduce the \textit{Pochhammer symbol} (or \textit{rising factorial}) defined by
\[(x)_j:=\prod_{k=0}^{j-1}(x+k)=\frac{\Gamma(x+j)}{\Gamma(x)},\qquad x\in\mathbb{R},\,j\in \mathbb{N}\cup\{0\},\]
where the expression in terms of the gamma function holds as long as $x$ and $x+j$ are not negative integers. Note that whenever we write a gamma function it is implicitly assumed that the argument is not a negative integer. 
The formula above has the obvious consequence that 
\begin{equation}\label{simpleformula}
\frac{(x+1)_j}{(x)_j} = \frac{x+j}{x}, \qquad x\in\mathbb{R}, \, j\in \mathbb{N}\cup\{0\}.
\end{equation}
Using this the \textit{hypergeometric function} $_2F_1(a,b;c;z)$ is then defined for $\vert z\vert<1$ and $c\notin \mathbb{Z^-}$ by the series
    \[_2F_1(a,b;c;z):=\sum_{j=0}^\infty\frac{(a)_j(b)_j}{(c)_j}\frac{z^j}{j!}\]
    and by analytic continuation elsewhere \cite{andrews1999special}. 
    The analytic continuation is defined by an integral representation of $_2F_1(a,b;c;z)$ due originally to Euler \cite{euler1794institutiones}.
For conciseness, we will denote it simply as $F$ from this point forward.

\begin{theorem}[Theorem 2.2.1 in \cite{andrews1999special}]\label{euler}
    If $\operatorname{Re}c>\operatorname{Re}b>0,$ then 
    \begin{equation}\label{eulerintegral}
        F(a,b;c;z)=\frac{\Gamma(c)}{\Gamma(b)\Gamma(c-b)}\int_0^1t^{b-1}(1-t)^{c-b-1}(1-zt)^{-a}\,dt
    \end{equation}
    in the $z$-plane with branch cuts at 1 and $\infty$. 
\end{theorem}

Euler's integral representation yields the following expression for $F$ when $z=1$, due originally to Gauss \cite{gauss1812}.

\begin{theorem}[Theorem 2.2.2 in \cite{andrews1999special}]\label{gauss}
    For $\operatorname{Re}(c-a-b)>0$, it holds that
    \[F(a,b;c;1)=\frac{\Gamma(c)\Gamma(c-a-b)}{\Gamma(c-a)\Gamma(c-b)}.\]
\end{theorem}

We now observe what happens if $a$ is a negative integer, i.e., $a=-m$ for $m\in\mathbb{N}$. In this case, the Pochhammer symbol $(a)_j$ vanishes for $j>m$, and hence the series terminates at $j=m$. Note also that 
\[(-m)_j=(-1)^j\,j!\binom{m}{j},\]
which yields the following result:
\begin{corollary} 
\label{truncatedhyper}
 If $a=-m$ for $m\in\mathbb{N}$, the hypergeometric function $F(a,b;c;z)$ reduces to the polynomial
 \[F(-m,b;c;z):=\sum_{j=0}^m(-1)^j\binom{m}{j}\frac{(b)_j}{(c)_j}z^j.\]
\end{corollary}

The following derivative formulas quickly yield antiderivative formulas used in Proposition \ref{greens} to derive the Green function for the inverse of $(-\Delta)^k$ on radial vector fields. 
\begin{proposition}\label{weightderivprop}
    For any real parameters $a$, $b$, $c$, and any complex $z$, we have 
    \begin{align}
        \frac{d}{dz}\big[ z^c F(a,b;c+1;z)\big] &= c z^{c-1} F(a,b;c;z), \label{type1deriv} \\
        \frac{d}{dz} \big[ z^{a-1} F(a-1,b;c;z)\big] &= (a-1) z^{a-2} F(a,b;c;z), \label{type2deriv}
        \\
        \frac{d}{dz} \big[ z^{b-1} F(a,b-1;c;z)\big] &= (b-1) z^{b-2} F(a,b;c;z). \label{type3deriv}
    \end{align}
\end{proposition}

\begin{proof}
For \eqref{type1deriv}, we have 
$$ \frac{d}{dz}\big[z^c F(a,b;c+1;z)\big] 
        = \sum_{j=0}^{\infty} \frac{(a)_j (b)_j}{j!} \, \frac{(c+j) z^{c+j-1}}{(c+1)_j},
        $$ and the result follows immediately from 
        \eqref{simpleformula}. 
 
  For \eqref{type2deriv}, we compute 
    $$ \frac{d}{dz} \big[ z^{a-1} F(a-1,b;c;z)\big] = \sum_{j=0}^{\infty} \frac{(b)_j}{(c)_j j!} \, (a+j-1) (a-1)_j \, z^{a+j-2},
$$   
and the formula follows from $(a+j-1)(a-1)_j = (a-1) (a)_j$, which is just a restatement of \eqref{simpleformula} for $x=a-1$. 
Formula \eqref{type3deriv} is a consequence of \eqref{type2deriv} since $F$ is invariant under switching $a$ and $b$. 
\end{proof}

Gauss defined two hypergeometric functions to be \textit{contiguous} if they are power-series in the same variable, if two of the parameters $\{a,b,c\}$ are pairwise equal, and if the third pair of parameters differ by 1. He showed \cite{gauss1812} that a hypergeometric function can always be written as a linear combination of any two others contiguous to it, where the coefficients are rational functions of $a,b,c,$ and $z$. 
One of these contiguous relations will be especially important later in the proof of Proposition \ref{greens}, so we derive it here. 
\begin{corollary}\label{contiguouscorollary}
    For any real parameters $a$, $b$, $c$ and any complex $z$, we have     
    \[(a-1)\,F(a,b-1;c;z)-(b-1)\,F(a-1,b;c;z)=(a-b)\,F(a-1,b-1;c;z).\]
\end{corollary}

\begin{proof} 
    Replacing $b$ with $(b-1)$ in \eqref{type2deriv}, the product rule and a cancellation gives  
    $$ z\, \tfrac{d}{dz} F(a-1,b-1;c;z) = (a-1) F(a,b-1;c;z) - (a-1) F(a-1,b-1;c;z). $$
    Similarly replacing $a$ with $(a-1)$ in \eqref{type3deriv} yields 
    $$ z\,\tfrac{d}{dz} F(a-1,b-1;c;z) = (b-1) F(a-1,b;c;z) - (b-1) F(a-1,b-1;c;z). $$
    Subtracting the two gives the desired formula.
 \end{proof}

The following bounds are an easy consequence of Euler's formula in Theorem \ref{euler}. They are crucial in the final proof of Theorem \ref{theorem:maintheorem}.

\begin{lemma}\label{boundslemma}
If $a\le 0$ with $c>b>0$ and $0\le z\le 1$, 
then $F(a,b;c;z)$ satisfies 
\begin{equation}\label{decreasinghypergeom}
\frac{d}{dz} F(a,b;c;z) \le 0
\end{equation}
and 
\begin{equation}\label{hypergeombounds}
0< \frac{\Gamma(c) \Gamma(c-b-a)}{\Gamma(c-b)\Gamma(c-a)} \le F(a,b;c;z) \le 1.
\end{equation}
\end{lemma}

\begin{proof}
    For the derivative bound, we differentiate Euler's integral \eqref{eulerintegral} and get 
    $$ \frac{d}{dz} F(a,b;c;z) = 
    \frac{a\Gamma(c)}{\Gamma(b)\Gamma(c-b)}\int_0^1 t^{b}(1-t)^{c-b-1}(1-zt)^{-a-1}\,dt.
    $$
    Since $c>b>0$ we know all the $\Gamma$ terms are positive. The integrand is also obviously positive since $0\le z\le 1$, and thus the factor of $a$ makes the entire thing nonpositive.

    For \eqref{hypergeombounds}, we simply use the fact that $F$ is decreasing on the interval to conclude that 
    $$ F(a,b;c;1)\le F(a,b;c;z)\le F(a,b;c;0),$$
    and the fact that $F(a,b;c;0)=1$ together with Gauss' formula from Theorem \ref{gauss}.
\end{proof}

\subsection{Green function for Laplace operators acting on radial vector fields}\label{sec:greensfunction}

To invert the integer-order homogeneous Sobolev inertia operator $A=(-\Delta)^k$, we will impose certain decay conditions as $r\to\infty.$ For $\lambda\in\mathbb{R},$ denote by $\mathcal{X}_{\lambda}$ the space of functions that decay at infinity like a power $r^{-\lambda}$:
\[\mathcal{X}_\lambda\left(\mathbb{R}_{\ge0},\mathbb{R}\right):=\left\{u:\mathbb{R}_{\ge0}\to\mathbb{R}\,\vert \,\limsup_{r\to\infty}\,r^{\lambda}\vert u(r)\vert <\infty\right\}.\]
In addition, we assume the vector fields $U=u(r)\partial_r$ can be extended to smooth vector fields on $\mathbb{R}^n,$ which leads to the condition that all even derivatives at $r=0$ vanish. Hence we define
\[W^{m,1}_{\text{odd}}\left(\mathbb{R}_{\ge0},\mathbb{R}\right):=\left\{ u\in W^{m,1}\left(\mathbb{R},\mathbb{R}\right)\,\vert\,u^{(2k)}(0)=0,\,0\le2k<m\right\},\]
where $W^{m,1}$ is the usual Sobolev space of locally integrable functions with locally integrable weak derivatives up to order $m$. We then take the intersection of these two spaces as the domain for the radial Laplace operator. That is,
\[\mathcal{Q}_{\lambda}^m\left(\mathbb{R}_{\ge0},\mathbb{R}\right):=\mathcal{X}_\lambda\left(\mathbb{R}_{\ge0},\mathbb{R}\right)\cap W^{m,1}_{\text{odd}}\left(\mathbb{R}_{\ge0},\mathbb{R}\right).\]
This functional setting allows us to iterate the solution formula for $-\Delta(u(r)\partial_r)=\omega(r)\partial_r$, which yields the following result concerning the Green function for the higher-order homogeneous operator.

\begin{proposition}\label{greens}
 Let $k\in\mathbb{N}$ with $2(k-1)<\lambda<n$.  Given any $\omega\in\mathcal{Q}_{2k+\lambda-1}^0\left(\mathbb{R}_{\ge0},\mathbb{R}\right),$ there exists a unique solution $u\in\mathcal{Q}_{\lambda-1}^{2k}\left(\mathbb{R}_{\ge0},\mathbb{R}\right)$ of $(-\Delta)^k\left(u(r)\partial_r\right)=\omega(r)\,\partial_r$ that takes the form 
\begin{equation}\label{solutionformula}
u(r)=\int_0^\infty K_k(r,s)s^{n-1}\omega(s)\,ds,\qquad K_k(r,s)=\delta_k(\min\{r,s\},\max\{r,s\}),
\end{equation} 
where $\delta_k(r,s)=rs\,\varphi_k(r,s),$ and $\varphi_k$ is defined on $D=\{(r,s)\,\vert\,s\ge r>0\}\subset\mathbb{R}^2$ and given by
\begin{multline}\label{varphik}
\varphi_k(r,s)= \Constanta(k,n) s^{2k-2-n} 
\, F\left(1-k,\frac{n}{2}+1-k;\frac{n}{2}+1;\frac{r^2}{s^2}\right), \\
\text{where} \quad \Constanta(k,n):= \frac{\Gamma\left(\frac{n}{2}+1-k\right)}{2^{2k-1}(k-1)!\, \Gamma\left(\frac{n}{2}+1\right)}.
\end{multline}
\end{proposition}

\begin{remark}
   Since $k\in\mathbb{N}$ in the previous proposition, the hypergeometric function in the expression for $\varphi_k(r,s)$ takes the form of a polynomial in the variable $z=r^2/s^2$ by Corollary \ref{truncatedhyper}. This leads, for example, to the formulae
   \begin{align*}
        \varphi_1(r,s)&=\frac{s^{-n}}{n}\\
        \varphi_2(r,s)&=\frac{s^{-n}}{2n}\Big(\frac{s^2}{n-2}-\frac{r^2}{n+2}\Big)\\
       \varphi_3(r,s)&=\frac{s^{-n}}{8n}\Big(\frac{s^4}{(n-4)(n-2)}-\frac{2r^2s^2}{(n-2)(n+2)}+\frac{r^4}{(n+2)(n+4)}\Big),\\
       \varphi_4(r,s)&=\frac{s^{-n}}{48n}\Big(\frac{s^6}{(n-6)(n-4)(n-2)}-\frac{3r^2s^4}{(n-4)(n-2)(n+2)}\\
       &\qquad\qquad\qquad+\frac{3r^4s^2}{(n-2)(n+2)(n+4)}-\frac{r^6}{(n+2)(n+4)(n+6)}\Big).
   \end{align*}
\end{remark}

\begin{proof}[Proof of Proposition~\ref{greens}] 
The existence of a unique solution $u\in\mathcal{Q}_{\lambda-1}^{2k}\left(\mathbb{R}_{\ge0},\mathbb{R}\right)$ to $(-\Delta)^k\left(u(r)\partial_r\right)=\omega(r)\,\partial_r$ for any $\omega\in\mathcal{Q}_{2k+\lambda-1}^0\left(\mathbb{R}_{\ge0},\mathbb{R}\right)$ is shown in \cite{bauer2024liouville}. We will establish the solution formula by induction on $k$. When $k=1$, the expression for $\varphi_k$ reads 
\[\varphi_1(r,s)=\frac{s^{-n}\,\Gamma\left(\frac{n}{2}\right)}{2\Gamma\left(\frac{n}{2}+1\right)}\,F \left(0,\frac{n}{2};\frac{n}{2}+1;\frac{r^2}{s^2}\right)=\frac{s^{-n}}{n}\]
with the last equality true since $\Gamma(x+1)=x\Gamma(x)$ and $F(0,b;c;z)=1$. The solution formula then becomes
\[u(r)=\frac{r^{1-n}}{n}\int_0^r s^{n}\omega(s)\,ds+\frac{r}{n}\int_r^\infty\omega(s)\,ds,\]
which one can directly check solves $\Delta\left(u(r)\partial_r\right)=-\omega(r)\,\partial_r$, noting that the vector Laplacian is given in Remark \ref{vectorlaplacian}. This establishes the base case. Now suppose that
\[u(r)=\int_0^\infty K_k(r,s)s^{n-1}\omega(s)\,ds\]
solves $(-\Delta)^k\left(u(r)\partial_r\right)=\omega(r)\,\partial_r$ for a fixed integer $k>1$.
Iterating, we see that 
\[u(r)=\int_0^\infty\left(\int_0^\infty K_k(r,\sigma)K_1(\sigma,s)\sigma^{n-1}\,d\sigma\right)s^{n-1}\omega(s)\,ds=\int_0^\infty K_{k+1}(r,s)s^{n-1}\omega(s)\,ds\]
solves $(-\Delta)^{k+1}\left(u(r)\partial_r\right)=\omega(r)\,\partial_r$, where 
\[K_{k+1}(r,s)=\delta_{k+1}(r,s)=rs\,\varphi_{k+1}(r,s)\quad\text{for}\quad r\le s\]
and 
\begin{equation}\label{induct}
\varphi_{k+1}(r,s)=\int_0^\infty \sigma^{n+1}\varphi_k(\min\{r,\sigma\},\max\{r,\sigma\})\varphi_1(\min\{\sigma,s\},\max\{\sigma,s\})\,d\sigma.
\end{equation}
We must show that this holds for $\varphi_k$ given by \eqref{varphik}, as long as $k+1<n/2+1$, i.e., $k<n/2$. Define the integral on the RHS to be $I$, so that
\begin{equation*}
    \begin{split}
        I&=\int_0^r\sigma^{n+1}\varphi_k(\sigma,r)\varphi_1(\sigma,s)\,d\sigma+\int_r^s \sigma^{n+1}\varphi_k(r,\sigma)\varphi_1(\sigma,s)\,d\sigma+\int_s^\infty \sigma^{n+1}\varphi_k(r,\sigma)\varphi_1(s,\sigma)\,d\sigma\\
       &=:I_1+I_2+I_3.
    \end{split}
\end{equation*}
We first compute $I_1$. Letting $a=1-k$, $b=n/2+1-k$, and $c=n/2+1$
yields
\[I_1=\int_0^r\sigma^{n+1}\varphi_k(\sigma,r)\varphi_1(\sigma,s)\,d\sigma
= \frac{s^{-n}}{n}\, \Constanta(k,n) r^{2k-2-n} \int_0^r \sigma^{n+1} F(a,b;c;\sigma^2/r^2) \, d\sigma.\]
Changing variables to $z=\sigma^2/r^2$, we get 
\[ I_1 = 
\frac{\Constanta(k,n)}{2n} \, s^{-n}r^{2k} \int_0^1 z^{c-1}\,F(a,b;c;z)\,dz.\]
We identify the integrand as the right side of \eqref{type1deriv}, and immediately obtain 
$$ I_1 = \frac{\Constanta(k,n)r^{2k}}{2cns^n} z^c F(a,b;c+1;z)\big|_{z=0}^{z=1} = 
\frac{\Constanta(k,n)}{2n(\tfrac{n}{2}+1)} \, \frac{r^{2k}}{s^n} \, F(a,b;c+1;1). $$

For $I_2$, change variables using $z=r^2/\sigma^2$ to find that
\begin{align*}
    I_2&=\int_r^s \sigma^{n+1}\varphi_k(r,\sigma)\varphi_1(\sigma,s)\,d\sigma = \frac{\Constanta(k,n)}{n} \, s^{-n} \int_r^s 
\sigma^{2k-1} F(a,b;c;r^2/\sigma^2) \, d\sigma \\
&= \frac{\Constanta(k,n)}{2n} \, \frac{r^{2k}}{s^n} \int_{r^2/s^2}^1 z^{a-2} F(a,b;c;z) \, dz. 
\end{align*}

Formula \eqref{type2deriv} now implies 
\begin{align*}
I_2 &= \frac{\Constanta(k,n)}{2n(a-1)} \, \frac{r^{2k}}{s^n} \, z^{a-1} F(a-1,b;c;z)\Big|_{z=r^2/s^2}^1 \\
&= -\frac{\Constanta(k,n)}{2nk} \, \Big( 
\frac{r^{2k}}{s^n} F(a-1,b;c;1) - s^{2k-n} F(a-1,b;c;r^2/s^2)\Big).
\end{align*}

Lastly, for $I_3$, we have
\[I_3=\int_s^\infty \sigma^{n+1}\varphi_k(r,\sigma)\varphi_1(s,\sigma)\,d\sigma= \frac{\Constanta(k,n)}{n} \int_s^{\infty} 
\sigma^{2k-1-n} F(a,b;c;r^2/\sigma^2) \, d\sigma.
\]
We change variables using $z=r^2/\sigma^2$ and apply \eqref{type3deriv} to obtain 
\begin{align*}
    I_3 &= \frac{\Constanta(k,n)}{2n} \, r^{2k-n} \int_0^{r^2/s^2} z^{b-2} F(a,b;c;z) \, dz = 
    \frac{\Constanta(k,n)}{2n(b-1)} \, r^{2k-n} z^{b-1} F(a,b-1;c;z) \Big|_{z=0}^{r^2/s^2} \\ 
    &= \frac{\Constanta(k,n)}{2n(\tfrac{n}{2}-k)} \, s^{2k-n} \, F(a,b-1; c, r^2/s^2).
\end{align*}
Note that in the last term we need to use the fact that $k<n/2$ so that $b-1>0$, in order to conclude that $z^{b-1} F(a,b+1;c;z)$ approaches zero as $z\to 0$.

Finally combining the three integrals gives 
\begin{multline*}
    I = I_1+I_2+I_3 = 
    \frac{\Constanta(k,n) \,r^{2k}}{2nc(a-1) s^n} \big[ (a-1) \, F(a,b;c+1;1) + c \, F(a-1,b;c;1)\big] \\ 
    + \frac{\Constanta(k,n)}{2n(a-1)(b-1)} \, s^{2k-n} \big[ (a-1)\, F(a,b-1;c;r^2/s^2) - (b-1) \, F( a-1, b; c; r^2/s^2)\big]
\end{multline*}

We can easily check that the first term disappears using Theorem \ref{gauss} and the fact that $a+c=b+1$. Meanwhile the second term reduces, by the contiguous relation of Corollary \ref{contiguouscorollary}, to    $$  I = \frac{\Constanta(k,n)(a-b)}{2n(a-1)(b-1)} \, s^{2k-n} \, F(a-1,b-1;c;r^2/s^2). $$
Since we can verify that
$$ \frac{\Constanta(k,n)(a-b)}{2n(a-1)(b-1)} = \Constanta(k+1,n),$$
this produces the formula 
$$ \varphi_{k+1}(r,s) 
= \Constanta(k+1,n) s^{2(k+1)-2-n} F(1-(k+1), \tfrac{n}{2}+1-(k+1); \tfrac{n}{2}+1; r^2/s^2),$$
which is the inductive step we wanted.
\end{proof}

\subsection{Proof of Breakdown of the EPDiff equation with homogeneous Sobolev inertia operator}\label{sec:blowup_hom}
To prove our main result, breakdown of the EPDiff equation with inertia operator $A=(-\Delta)^k$, we will employ the comparison based breakdown theorem of Section~\ref{generalbreakdown} and exploit the lower bound on hypergeometric functions from Section~\ref{sec:hypergeometric}.

\begin{proof}[Proof of Theorem~\ref{theorem:maintheorem}] 
The following proof will work for integers $k\ge1.$ For completeness, we include the case $k=0$ in the statement of the theorem; this case was already established in \cite{chae2012blow}.  We will now proceed with the remaining $k\ge1$. To apply Theorem~\ref{generalbreakdown} we need to choose a comparison function $Q$, which we will choose as described in~Remark \ref{canonicalQ}. 
 Using the expression for the auxiliary function $\varphi_k$ provided by Proposition \ref{greens} yields for $\varphi_k(0,r)$ the expression
 \[\varphi_k(0,r)=C(k,n)r^{2k-2-n},\]
 which converts the breakdown conditions \eqref{condition1} and \eqref{condition2} to
\begin{equation}\label{condition1Hk}
 \Psi_k(s,r):=(n+2-2k)\varphi_k(s,r) +r\frac{\partial\varphi_k}{\partial r}(s,r)\geq0\quad\text{for all}\quad s\in[0,r),
\end{equation}
\begin{equation}\label{condition2Hk}
     \tilde{\Psi}_k(r,s):=(n+2-2k) \varphi_k(r,s) + r\,\frac{\partial \varphi_k}{\partial r}(r,s)\geq\frac{C}{s^{n-2k+2}} \quad\text{for all}\quad s\in[r,\infty).
\end{equation} 
    We will establish each of the inequalities \eqref{condition1Hk}-\eqref{condition2Hk} in the following two cases. First recall the formula for $\varphi_k$, which we write in the form 
    $$ \varphi_k(r,s) = \Constanta(k,n) s^{-2b} F(a,b;c;r^2/s^2), \quad a=1-k, b=\tfrac{n}{2}+1-k, c=\tfrac{n}{2}+1,$$
    as in the proof of Proposition \ref{greens}.
\begin{enumerate}
\item We first show that $\Psi_k(s,r)\ge0$ for $s\in[0,r).$ 
\begin{align*} 
\Psi_k(s,r) &= \Constanta(k,n) \Big( 2b r^{-2b} F(a,b;c;s^2/r^2) + r \, \frac{\partial}{\partial r}\big[ r^{-2b} F(a,b;c;s^2/r^2)\big] \Big) \\
&= -2 \Constanta(k,n) r^{-2b-1} s^{2}\, \frac{d}{dz}\Big|_{z=s^2/r^2} F(a,b;c;z).
\end{align*}
By the inequality \eqref{decreasinghypergeom}, this is nonnegative since $z=s^2/r^2\in [0,1]$. Note that the condition $1\leq k< \frac{n}{2}+1$ ensures that 
$a\leq 0$ as required. 
\item Next we compute 
\begin{align*}
    s^{2b} \tilde{\Psi}_k(r,s) &= \Constanta(k,n) \Big[ 2b F(a,b;c;r^2/s^2) + r \, \frac{\partial}{\partial r} F(a,b;c;r^2/s^2) \Big] \\
    &= 2\Constanta(k,n) \Big[ 
    b F(a,b;c;z) + z \, \frac{d}{dz} F(a,b;c;z)\Big]\Big|_{z=r^2/s^2} \\
    &= 2\Constanta(k,n) z^{1-b}\Big|_{z=r^2/s^2} \, \frac{d}{dz}\Big|_{z=r^2/s^2} \Big( z^b F(a,b;c;z)\Big) \\
    &= 2b \Constanta(k,n) F(a,b+1;c;r^2/s^2), 
\end{align*}
using formula \eqref{type3deriv} with $b$ replaced by $(b+1)$. By inequality \eqref{hypergeombounds}, this is bounded below by a positive constant since $0\le r^2/s^2\le 1$, which is precisely condition \eqref{condition2Hk}. Here we use in addition that $b>0$, which is guaranteed since $k< \frac{n}{2}+1$. 
\end{enumerate}

By Remark \ref{canonicalQ}, and ultimately Theorem \ref{generalbreakdown}, this completes the proof.
\end{proof}

\subsection{Towards Conjecture~\ref{conjecture}: Blowup for non-homogenous and non-integer order inertia operators}\label{sec:future}
Here we describe how one should be able to extend the results of the current analysis for the operator $A=(-\Delta)^k$ for integer $k$ to the more general case $A=(\sigma - \Delta)^k$ for any real $k\in [0,n/2+1)$.
We first note that our current approach depends on an interpretation of the EPDiff equation as an ODE on an infinite-dimensional Banach space of functions. However, this interpretation is only valid for $k\geq \frac12$, and thus we expect that an entirely different method will be necessary to obtain the break-down for small $k<\frac12$; here we note that this is also the case for $k=0$, corresponding to Burgers' equation, for which our blowup proof would not be applicable. The primary challenge in achieving the result for 
$k\geq \frac12$, where the overall approach of this paper is expected to be applicable, lies in deriving a suitable formula for the Green function: while the equation presented in Proposition~\ref{greens} is formally well-defined for non-integer $k$, we do not know if this corresponds to the Green function of the fractional Laplacian. 
In addition, the situation for $\tfrac{1}{2}\le k<1$ will be somewhat different since the  current proof does not work, even if we had the correct Green function: the first condition in the blowup criterion is not true in that case, and a different comparison function will be needed. Additional difficulties arise when dealing with the non-homogeneous operator $A=(1-\Delta)^k$: while we do not believe that the lower-order terms should have an effect on the existence of initial conditions that admit solution breakdown, they significantly complicate the formula for the Green function; for this case it is not yet clear what the formula for the Green function would look like even for general positive integers $k$. It would also be interesting to study the critical case when $k=n/2+1$. Based on the results in one dimension~\cite{preston2018euler, bauer2020geodesic}, we would guess that global existence of smooth solutions is true in any dimension. Finally, these breakdown results are only presented on the domain $\mathbb{R}^n$ where the Green function is simpler and the equation admits radial solutions. This is used in an essential way since the negative momentum $\omega_0$ drives the flow map derivative $\gamma_r(t,0)$ to zero at the origin, based on an influence from arbitrarily large $r$. It is not clear whether the same breakdown results hold in general on the periodic domain $\mathbb{T}^n$, or on the ball in $\mathbb{R}^n$ with some boundary conditions.

\bibliographystyle{abbrv}
\bibliography{refs}

\end{document}